\documentclass[11pt,a4paper]{article}

\usepackage[utf8]{inputenc}
\usepackage[T1]{fontenc}
\usepackage[a4paper,margin=2.6cm]{geometry}

\usepackage{amsmath,amssymb,amsthm,amsfonts}
\usepackage{latexsym}
\usepackage{footnote}
\usepackage{multicol}
\usepackage{booktabs}
\usepackage[flushleft]{threeparttable}
\usepackage[font=small,labelfont=bf]{caption}
\usepackage[hidelinks]{hyperref}

\theoremstyle{plain}
\newtheorem{theorem}{Theorem}[section]
\newtheorem{corollary}{Corollary}[section]
\newtheorem{lemma}{Lemma}[section]

\theoremstyle{definition}
\newtheorem{definition}{Definition}[section]

\theoremstyle{remark}
\newtheorem{remark}{Remark}[section]

\makeatletter
\newcommand{\hinfo}[7]{}
\newcommand{\hinfoo}[2]{}

\newcommand{\coraut}{\textsuperscript{$\ast$}}
\def\@addr{}\def\@emails{}
\newcommand{\address}[1]{\gdef\@addr{#1}}
\newcommand{\emails}[1]{\gdef\@emails{#1}}
\newcommand{\subjclass}[1]{\par\medskip\noindent
  \textbf{2020 Mathematics Subject Classification.} #1\par}
\newcommand{\keywords}[1]{\par\smallskip\noindent
  \textbf{Keywords.} #1\par}

\renewcommand{\maketitle}{%
  \begin{center}
    {\LARGE\bfseries \@title \par}\vspace{1.1em}
    {\large \@author \par}\vspace{0.7em}
    {\small \@date \par}%
  \end{center}\vspace{0.6em}
}
\newcommand{\AuthorAddresses}{%
  \par\vspace{1.4em}\noindent\hrulefill\par\vspace{0.5em}
  {\footnotesize\noindent \@addr \par\vspace{0.4em}
   \textit{E-mail addresses:}\ \@emails \par\vspace{0.3em}
   \textsuperscript{$\ast$}\,\textit{Corresponding author.}\par}%
}
\makeatother
\date{\today}
\begin{document}

\markboth{H. Belbachir, H. Zeggada}{The distribution of the de Moivre experiment}

\title{The distribution of the de Moivre experiment}

\author{Hac\`ene Belbachir\coraut$^{1}$, Hamza Zeggada$^2$}

\address{$^1$USTHB, Faculty of Mathematics, RECITS Laboratory,\\ Po.~Box 32, El Alia 16111, Bab Ezzouar, Algiers, Algeria\\
	$^2$USTHB, Faculty of Mathematics, RECITS Laboratory,\\ Po.~Box 32, El Alia 16111, Bab Ezzouar, Algiers, Algeria\\}
\emails{hbelbachir@usthb.dz, hacenebelbachir@gmail.com (H.~Belbachir); hzeggada@usthb.dz (H.~Zeggada)}
\maketitle
\begin{abstract}
In this paper, we focus on the de Moivre random experiment, which allows us to introduce the $s$-Bernoulli distribution and the bi$^{s}$nomial distribution. We present some probabilistic properties such as the expectation, the variance, the skewness and kurtosis coefficients, the moments, the cumulants and the generating functions; the moments are expressed through extended Eulerian numbers and partial Bell polynomials, and the cumulants through logarithmic polynomials. Then we establish that, for an integer $s\geq 2$, the bi$^{s}$nomial distribution converges to a Poisson limiting distribution when the expectations converge, and that, for fixed parameters with $p\neq q$, it satisfies a central limit theorem as $n\rightarrow\infty$.
\end{abstract}

\section{Introduction}
In 1712, Abraham de Moivre discussed in \emph{De Mensura Sortis} \cite{Dem2} (see also \cite[p.~39--43]{Dem}), through different examples, the bi$^{s}$nomial coefficients. The main result appeared in a lemma which stated: \textit{``To find how many chances there are upon any number of dice, each of them of the same given number of faces, to throw any given number of points''} \cite[p.~39]{Dem}. If $n$ denotes the number of dice, $(s+1)$ the number of faces and $k$ the number of points, then this number of chances is denoted by $ {n \choose k}_s. $
Several extensions and studies have been investigated in the literature; for example, Bondarenko \cite{Bon} gave a combinatorial interpretation of bi$^{s}$nomial coefficients $ {n\choose k}_s $ as \textit{the number of different ways of distributing $ k $ balls among $ n $ cells where each cell contains at most $ s $ balls.} If we denote by $ x_i $ the number of balls in the $i$-th cell, the interpretation given by Bondarenko is equivalent to evaluating the number of solutions of the system
\begin{equation}
	\left\lbrace \begin{array}{ll}
		x_1+\cdots+x_n=k,\\
		0\leq x_1,\ldots,x_n\leq s.
	\end{array}\right. 
	\label{eq 1}
\end{equation}

Belbachir, Bouroubi and Khelladi \cite{BBe13} obtained an explicit relation with one summation symbol to compute the bi$^{s}$nomial coefficients:
\begin{equation}
	{n\choose k}_s = \sum_{j=0}^{\lfloor k/(s+1)\rfloor}(-1)^j {n\choose j} {k-j(s+1)+n-1\choose n-1}.
	\label{eq 2}
\end{equation}

In fact, de Moivre (see \cite{Dem2}) solved the system \eqref{eq 1} by giving the right hand side of \eqref{eq 2}.

This alternating expression is difficult to handle in probability, where expressions involving only positive terms are preferred. We therefore start from this problem, which interested the community at the time but could not be modeled, since simplified expressions became available only in the 2000s.

In the present work, the random experiment on which we build our distribution is de Moivre's experiment \cite{Dem2}, where he stated the following probability problem: ``there are $ n $ dice with $ (s + 1) $ faces; if these are thrown randomly, what would be the chance of the exhibited sum of numbers to be equal to $ k $?''

Since no such model was available at the time, the resulting probability distribution has, to the best of our knowledge, rarely been considered in the literature; a related construction, in the classical case $p+q=1$, was recently proposed by Atta \cite{Atta} under the name of $p$-nomial distribution. The case where the number of faces equals $2$, i.e.\ the coin toss problem, is the origin of the Bernoulli and binomial distributions. It corresponds to the modeling of a uniform distribution on $ \left\lbrace 0,1\right\rbrace $; the modeling of the uniform distribution on $ \left\lbrace 0,\ldots,s\right\rbrace  $, and of its exponential tilting, has rarely been developed in the literature. We propose in this paper to develop it, as well as the distribution laws associated with it. Let us point out that the one-die distribution introduced below is, up to reparametrization, a truncated geometric distribution on $\{0,\dots,s\}$; the novelty of the present approach lies in the parametrization $(p,q)$ under the normalization constraint, in the closed-form moment and cumulant identities through extended Eulerian numbers, logarithmic polynomials and partial Bell polynomials, and in the Poisson limit under the coupled constraint.

The experiment consists of tossing a die with $ (s+1) $ faces numbered by $\{0,1,\ldots,s\} $. Let $ p $ and $ q $ denote two real numbers of $]0,1[$, playing the roles of success and failure weights respectively and linked by the normalization constraint $\sum_{k=0}^{s}p^kq^{s-k}=1$; the scheme of the different possible outcomes is as follows:
\begin{equation}
	\begin{array}{|c|c|c|c|c|c|}
		\hline
		\mbox{Face ($ k $)}& 0 & 1 & 2&\cdots&s\\ \hline
		\text{Probability }(p_k)& q^s & p q^{s-1} & p^2q^{s-2}&\cdots&p^s\\ \hline
	\end{array}	
\end{equation}
Observe that $p_k$ is proportional to $(p/q)^k$: the law of a single face is a truncated geometric distribution of ratio $\theta=p/q$, and the family is a one-parameter exponential family.

The bi$^{s}$nomial distribution, like the classical binomial distribution, converges to Poisson and normal limiting distributions; in this work we provide proofs of both convergences.

The Poisson distribution was introduced in 1837 by Sim\'eon Denis Poisson, in his work \emph{Recherches sur la probabilit\'e des jugements en mati\`ere criminelle et en mati\`ere civile} \cite{Pois}. In particular, in Chapter 8, he obtained his distribution as the limit of the binomial distribution $ B (T, \lambda / T) $ when $ T $ becomes large, but moved on fairly quickly to other considerations; he did not study this limit distribution. In 1712, de Moivre had obtained the same distribution in \cite{Dem2}, still as the limit of the binomial distribution. This has prompted some authors \cite{Stig,Hald} to argue that the Poisson distribution should bear the name of de Moivre.

On the other hand, the first appearance of a normal distribution was in 1733 by Abraham de Moivre, who obtained it by deepening the study of the factorial $ n! $ while analyzing a coin toss game, in a supplement to his \emph{Miscellanea analytica} \cite{Dem1}; he later included it in \emph{The Doctrine of Chances} \cite{Dem}, in which a normal distribution appears as the limit of a binomial distribution, which would later be at the origin of the central limit theorem. In 1777, Pierre-Simon Laplace resumed this work and obtained a good approximation of the error between this normal distribution and a binomial distribution using Euler's gamma function \cite{Bru}.

A normal distribution was then fully defined when the first central limit theorem, then called Laplace's theorem, was stated by Laplace in 1812 \cite{Bru}. The name ``normal'' spread at the end of the nineteenth century; according to Stigler \cite{Stig}, it was introduced independently by Peirce, Galton and Lexis, and popularized in France by Henri Poincar\'e.

In what follows, after having presented the de Moivre experiment, we define in the next sections the $ s $-Bernoulli and bi$^{s}$nomial distributions respectively with their properties, including their moments and cumulants; then in the fifth and sixth sections we obtain the Poisson distribution and the normal distribution respectively as limiting cases of the bi$^{s}$nomial distribution.

\section{Bi$ ^s $nomial Coefficients}

The coefficients $ { n \choose k}_s $ were introduced first in the work of Abraham de Moivre \cite{Dem2} and later Euler \cite{Eul,Eul1} studied these coefficients and derived a number of properties; some results were developed by Euler using the construction of the $ s $-Pascal triangle, which is a natural generalization of the Pascal triangle. In 1876, Andr\'e \cite{And} studied these coefficients using combinations on words and explored several properties; he gave the absorption identity of bi$^{s}$nomial coefficients \cite{And}, which was generalized by Belbachir and Igueroufa \cite{Iguer}. The first rows of the $s$-Pascal triangles appear in the On-Line Encyclopedia of Integer Sequences \cite{Slo} as sequences A027907 ($s=2$) and A008287 ($s=3$).

These coefficients have appeared in history under different names. We can cite for example: generalized binomial coefficients in 1993 by Bondarenko \cite{Bon}, ordinary multinomial coefficients in 1997 by Warnaar \cite{Wa} and in 1998 by Ericksen \cite{Er}. In our work we use the bi$^{s}$nomial terminology, which was introduced in \cite{BBBen14}. Using the classical binomial coefficient, one has the de Moivre alternating summation \eqref{eq 2}.

\section{The $ s $-Bernoulli distribution}

We define a random variable $ X $ taking its values in $X(\Omega)= \{0,1,2,\ldots,s\} $. We call $ X $ a random variable following \textbf{the $ s $-Bernoulli distribution} with parameter $ p $, denoted by $ X\sim \mathcal{B}^{(s)}(p)$. We define the probability distribution of $ X $ as follows.
\subsection{Probability distribution}
The probability distribution is given by
\begin{equation}
	\mathbb{P}(X=k)=\left\{{\begin{array}{llll}p^s&\quad {\mbox{for }}k=s,\\p^{s-1}q&\quad {\mbox{for }}k=s-1,\\ \vdots\\q^s&\quad {\mbox{for }} k=0,\end{array}}\right. \label{eq:pmf-sB}
\end{equation}
with $ \sum_{k=0}^{s}\mathbb{P}(X=k)=1 $, that is,
\begin{equation}
	\frac{p^{s+1}-q^{s+1}}{p-q} =1,  \text{ equivalently }  p\left( p^s-1\right) =q\left( q^s-1\right),   \text{ for }  p\neq q, \label{eq:constraint}
\end{equation}
and
\begin{equation}
	p = (s+1)^{-1/s}, \text{ for } p=q.
\end{equation}
For every ratio $\theta=p/q>0$ there exists a unique admissible pair $(p,q)\in\,]0,1[^2$ satisfying \eqref{eq:constraint}.

\subsection{The generating function}
In the following theorem, we present the generating function of the $ s $-Bernoulli distribution.
\begin{theorem}\label{thm:gf-sB}
	Let $ X $ be a random variable that follows an $s$-Bernoulli distribution with parameter $ p $, $ X \sim \mathcal {B}^{(s)} (p) $. Then the probability generating function of $X$ is, for $ p\neq q $,
	\begin{equation}
		G_X(t)=E(t^X) =  \sum_{k=0}^{s} t^k p^k q^{s-k} =\frac{(pt)^{s+1}-q^{s+1}}{pt-q} \qquad (pt\neq q,\ \text{removable singularity}),\label{eq:5}
	\end{equation}
	and
	\begin{equation}
		G_X(t)=\frac{1}{s+1}\left(\frac{1-t^{s+1}}{1-t}\right), \qquad\text{ for } p=q. \label{eq:5.1}
	\end{equation}
\end{theorem}
\begin{corollary}\label{cor:EV-sB}
	Let $ X $ be a random variable that follows an $s$-Bernoulli distribution with parameter $ p $, $ X \sim \mathcal {B}^{(s)} (p) $. Then the expectation and the variance of $ X $ are
	\begin{equation}
		E(X) =	\frac{p \left((s+1) p^{s}-1\right)}{(p-q)},\qquad V(X) =\frac{p q\left( 1-(s+1)^2 (p q)^s\right) }{(p-q)^2}, \text{ for } p\neq q, \label{eq:6}
	\end{equation}
	and
	\begin{equation}
		E(X)= s/2,\qquad V(X)=s(s+2)/12, \qquad \text{ for }p=q.
	\end{equation}
\end{corollary}

\begin{proof}
	Using the generating function given in \eqref{eq:5}, the expectation of $ X $, when $ p\neq q $, is
	\begin{equation*}
		E(X)=G'_X(1)= \frac{(s+1) p^{s+1}}{p-q}-\frac{p \left(p^{s+1}-q^{s+1}\right)}{(p-q)^2}=\frac{p\left((s+1)p^s-1\right)}{p-q},
	\end{equation*}
	where the last equality uses the normalization $p^{s+1}-q^{s+1}=p-q$ of \eqref{eq:constraint}. Similarly, $V(X)=G''_X(1)+G'_X(1)-G'_X(1)^2$ yields, after simplification through \eqref{eq:constraint}, the stated variance. The case $p=q$ is a direct computation from \eqref{eq:5.1}.
\end{proof}

\subsection{Relation between Eulerian numbers and the moment function}

In order to introduce higher-order moments, we need to consider the moment generating function of the $s$-Bernoulli distribution,
\begin{equation}
	M_X(t)=E(e^{tX})=\frac{(p e^t)^{s+1}-q^{s+1}}{pe^t-q}.
\end{equation} 

The explicit expression of this function reveals a relation with the explicit formula for the Eulerian numbers (see \cite{com1}), given by
\begin{equation}
	A(n,k) =\sum _{j=0}^k (-1)^{j} \binom{n+1}{j} (k+1-j)^n,\qquad \ 0\leq k\leq n-1,\label{eq19}
\end{equation}
where $A(n,k)$ counts the permutations of $\{1,\dots,n\}$ with exactly $k$ descents.\\
By introducing the parameter $ s $, we get an extended Eulerian number form.
\begin{definition} For integers $ n\geq0 $, $ k\geq0 $ and $ s\geq0 $, the extended Eulerian number is defined by
	\begin{equation}
		A(n,k,s) = \sum _{j=0}^k (-1)^{j} \binom{n+1}{j} (k+1-j-s)^n, \qquad \ 0\leq k\leq n. \label{eq 23}
	\end{equation}	
	For $ s=0 $ the summand reduces to that of \eqref{eq19}, so we recover the classical Eulerian numbers on their support, $ A(n,k,0)=A(n,k) $ for $ 0\leq k\leq n-1 $, with the additional value $ A(n,n,0)=0 $. Note that, contrary to the classical Eulerian numbers, the numbers $A(n,k,s)$ are not positive in general for $s\geq1$.
\end{definition}

\begin{theorem}\label{thm:moments-sB}
	The $r$-th moment $m_r(s)=E(X^r)$ of the $s$-Bernoulli distribution is given by
	\begin{equation}
		m_r(s) =\frac{(-1)^r }{(p-q)^{r+1}} p  \sum_{k=0}^r\left( A(r,k,s+1)  p^{r-k+s} q^{k} - A(r,k,1)  p^{r-k} q^{k+s} \right),\quad\text{ for } p\neq q,\label{eq:mr}
	\end{equation}
	and
	\begin{equation}
		m_r(s) = \frac{1}{(r+1)(s+1)}\sum_{k=0}^r(-1)^k { r+1 \choose k} B_k s^{r-k+1} , \quad\text{ for } p= q,
	\end{equation}
	where the $ B_k $ are the Bernoulli numbers, with the convention $B_1=-\tfrac12$; see \cite{Sriva}.
\end{theorem}
\begin{proof} 
	For $ p\neq q $, we proceed by induction on $r$: differentiating $M_X$ successively, the numerator of $M_X^{(r)}(t)$ splits into two polynomial forms in $pe^t$ and $q$, whose coefficients satisfy the recurrence of the extended Eulerian numbers $A(r,k,s)$; identifying the coefficients at $t=0$, the first form is generated by the coefficients $ A(r,k,s+1)  $ and the second by the extended Eulerian numbers $ A(r,k,1) $, which gives \eqref{eq:mr}. For $ p= q $, we have $p^s=\frac{1}{s+1}$ by \eqref{eq:constraint} and we find the Faulhaber-like formula $ m_r(s)=p^s\sum_{k=1}^s k^r  $, which can be expressed in terms of the Bernoulli numbers.
\end{proof}
We denote by $ \mu_r=E\left[(X-E(X))^r\right] $ and $ \sigma=\sqrt{\mu_2} .$ Then the skewness and the kurtosis are defined respectively by $ \gamma_1=\mu_3/\sigma^3 $ and $ \gamma_2=\mu_4/\sigma^4 $ (so that the excess kurtosis is $\gamma_2-3$).

We take as an example the third and fourth moments for any $ s $, when $ p=q $, which are respectively
\begin{equation}
	m_3(s)=\frac{1}{4} s^2 (s+1), \qquad m_4(s)=\frac{1}{30} s \left(6 s^3+9 s^2+s-1\right),
\end{equation}
where the skewness is zero and the kurtosis is as follows
\begin{equation}
	\gamma_2=\frac{3 \left(3 s^2+6s-4\right)}{5 s (s+2)}.
\end{equation}

\section{The bi$ ^s $nomial distribution}

The bi$^{s}$nomial distribution is obtained by repeating $ n $ times the de Moivre experiment, identically and independently. Equivalently, instead of a single die we take $ n $ dice with $ s+1 $ faces and observe a random variable $ X $ equal to the sum of the outcomes, which takes its values between $0$ and $sn$. This random variable follows the bi$^{s}$nomial distribution with parameters $ (n, p) $.

We define a random variable $ X $ taking its values in the set $ \{0,1,2, \ldots, sn \} $. We call $ X $ a variable following \textbf{the bi$^s$nomial distribution} with parameters $ (n, p) $, denoted by $ X\sim \mathcal{B}^{(s)}(n,p) $, with mass function
\begin{equation}
	\mathbb  {P}(X=k)= { n \choose k}_s p^kq^{sn-k},\quad 0\leq k\leq sn, \label{26}
\end{equation}
under the condition $ \sum_{k=0}^{s} p^k q^{s-k}=1 $; indeed,
$\sum_{k=0}^{sn}{n\choose k}_sp^kq^{sn-k}=\big(\sum_{k=0}^{s}p^kq^{s-k}\big)^n=1$.
\begin{itemize}
	\item For $ s=1 $, we find the classical binomial probability mass function.
	\item For $ p=q $, we have the uniform case, the $s$-uniform distribution:
	\begin{equation}
		\mathbb{P}(X=k)=\dfrac{{n \choose k}_s}{(s+1)^{n}}.
	\end{equation} 
\end{itemize}

In \cite{B}, Belbachir determined the smallest mode of the bi$^{s}$nomial coefficients, that is, for $n\geq2$, the smallest index achieving
\begin{equation}
	k_n:=\arg\underset{k}{\max}\,{n\choose k}_s=\left\lfloor \frac{sn}{2}\right\rfloor,
\end{equation}
which, in the uniform case $p=q$, is the mode of the distribution and yields the maximal probability $\mathbb{P}(X=k_n)$. (For $n=1$ the single die is uniform on $\{0,\dots,s\}$, so every index is a mode.)

\subsection{The generating function}
In the following theorem, we present the generating function of the bi$^{s}$nomial distribution.	
\begin{theorem}
	Let $ X $ be a discrete random variable with the bi$^{s}$nomial distribution with parameters $ n $ and $ p $, $ X\sim \mathcal{B}^{(s)}(n,p) $. Then, since $X$ is a sum of $n$ i.i.d.\ $s$-Bernoulli variables, $G_X=(G_{X_1})^n$; explicitly,
	\begin{equation}
		G_X(t) =\left(\frac{q^{s+1}-(pt)^{s+1} }{q-p t}\right)^n=\sum_{k=0}^{sn}{n\choose k}_s (pt)^kq^{sn-k},\quad \text{ for } p\neq q,\label{eq:19}
	\end{equation}
	and
	\begin{equation}
		G_X(t) =\frac{1}{\left( s+1\right)^n}\left( \frac{1-t^{s+1}}{1-t}\right)^n=\frac{1}{\left( s+1\right)^n}\sum_{k=0}^{sn}{n\choose k}_s t^k, \quad \text{ for } p=q.
	\end{equation}
\end{theorem}
\begin{corollary}
	Let $ X $ be a discrete random variable with the bi$^{s}$nomial distribution with parameters $ n $ and $ p $, $ X\sim \mathcal{B}^{(s)}(n,p) $. The expectation and variance are given by
	\begin{equation}
		E(X) =np\frac{(s+1)p^{s}-1}{(p-q)},\qquad V(X) =npq\frac{ 1-(s+1)^2(pq)^s}{(p-q)^2}, \quad \text{ for } p\neq q,
	\end{equation}
	and
	\begin{equation}
		E(X) =sn/2, \qquad V(X) = sn(s+2)/12,\quad \text{ for } p= q.
	\end{equation}
\end{corollary}	

\begin{proof}
	Since $X$ is a sum of $n$ i.i.d.\ $s$-Bernoulli random variables, $G_X=(G_{X_1})^n$; hence $E(X)=nE(X_1)$ and $V(X)=nV(X_1)$, and the result follows from Corollary~\ref{cor:EV-sB}. Alternatively, differentiate \eqref{eq:19} twice at $t=1$.
\end{proof}

\subsection{Relation between the partial Bell polynomials and the moment function}
The calculation of higher moments for the bi$^{s}$nomial distribution shows a connection to the partial Bell polynomials of $ x_1, x_2,\ldots $, defined by their generating function as follows (see \cite{com1}):
\begin{equation}
	\sum_{n=k}^{\infty}B_{n,k}(x_1,x_2,\ldots)\frac{t^{n}}{n!}=\frac{1}{k!}\left( \sum_{m=1}^{\infty}x_m
	\frac{t^m}{m!}\right)^k. \label{eq:bellgf}
\end{equation}
The explicit formula for $ B_{n,k}(x_1,x_2,\ldots) $ is given by
\begin{equation}
	B_{n,k}(x_{1},x_{2},\ldots )=\sum_{\substack{ j_{1}+j_{2}+\cdots =k\\ j_{1}+2j_{2}+\cdots=n}} \frac{n!}{j_{1}!j_{2}!\cdots }\left(\frac{x_{1}}{1!}\right)^{j_{1}} \left(\frac{x_{2}}{2!}\right)^{j_{2}}\cdots.\label{31}
\end{equation}
We note that the formula \eqref{31} involves a finite number of terms owing to the constraint $ j_{1}+2j_{2}+3j_3+\cdots=n .$ Thereafter we can use either notation $ B_{n,k}(x_{1},x_{2},\ldots ) $ or $ B_{n,k}(x_{1},x_{2},\ldots,x_{n-k+1} ) $; for more details we refer to \cite{com1}.
\begin{theorem}\label{thm:moments-bis}
	The moment of order $r$ of the bi$^{s}$nomial distribution, denoted by $M(r)$, has a polynomial form generated by the partial Bell polynomials:
	\begin{equation}
		M(r)=\sum_{k=1}^{r} (n)_k B_{r,k}\Big(\sum_{l=1}^{s}l p^l q^{s-l},\sum_{l=1}^{s}l^2 p^l q^{s-l},\ldots,\sum_{l=1}^{s}l^{r-k+1} p^l q^{s-l}\Big), \quad \text{ for } p\neq q,
	\end{equation} 
	and
	\begin{equation}
		M(r)=\sum_{k=1}^{r} (n)_k p^{ks} B_{r,k}\Big(\sum_{l=1}^{s}l ,\sum_{l=1}^{s}l^2,\ldots,\sum_{l=1}^{s}l^{r-k+1} \Big),\quad \text{ for } p= q,
	\end{equation}
	where $ (n)_k $ is the falling factorial, $ (n)_k=n(n-1)\cdots(n-k+1). $
\end{theorem}
\begin{proof}
	Write $M_X(t)=\big(M_{X_1}(t)\big)^n=\big(1+\sum_{i\geq1}x_i t^i/i!\big)^n$ with $ x_i = \sum_{l=1}^{s} l^{i}p^l q^{s-l}=E(X_1^i)  $. Expanding by the binomial theorem, $\big(1+u\big)^n=\sum_k\binom nk u^k$, and using \eqref{eq:bellgf} for $u^k/k!$, the coefficient of $t^r/r!$ is $\sum_{k=1}^{r}\binom nk k!\,B_{r,k}(x_1,x_2,\ldots)=\sum_{k=1}^{r}(n)_k B_{r,k}(x_1,x_2,\ldots)$, which is the expression of the bi$^{s}$nomial distribution moments. On the other hand, taking  $ x_i = p^{s}\sum_{l=1}^{s} l^{i}  $ we get the $ r $-th moment in the case $ p=q $, using the homogeneity $B_{r,k}(p^sa_1,p^sa_2,\ldots)=p^{ks}B_{r,k}(a_1,a_2,\ldots)$ of degree $k$ of the partial Bell polynomials.
\end{proof}
For $ p=q $, the third and fourth moments are
\begin{equation*}
	\begin{split}
		M(3)&= n^2 s^2 (n s+s+2)/8,\\        
		M(4)&=sn\left( 15 n^3 s^3+30 n^2 s^2 (s+2)+5 sn (s+2)^2-2 (s^3+4s^2+6s+4)\right)/240.
	\end{split}
\end{equation*}

Also the skewness is zero and the kurtosis is as follows
\begin{equation}
	\gamma_2=\frac{3((5 n-2) s (s+2)-4)}{5sn (s+2)}. \label{eq:kurt-bis}
\end{equation}
Observe that $\gamma_2\rightarrow3$ as $n\rightarrow\infty$, consistently with the normal convergence established in Section~6.

\subsection{The cumulants of the bi$^{s}$nomial distribution}
The cumulant generating function of $X\sim\mathcal{B}^{(s)}(n,p)$ is $K_X(t)=\ln M_X(t)=n\ln M_{X_1}(t)$: contrary to the moments, the cumulants $\kappa_r(X)=K_X^{(r)}(0)$ are \emph{exactly linear} in $n$,
\begin{equation}
	\kappa_r(X)=n\,\kappa_r(X_1),\qquad r\geq1,
\end{equation}
which makes them particularly convenient for asymptotic purposes. They are generated by the logarithmic polynomials (see \cite[Chap.~3]{com1}), i.e.\ by the alternating combination of partial Bell polynomials appearing in the expansion of $\ln\big(1+\sum_{m\geq1}x_mt^m/m!\big)$.

\begin{theorem}\label{thm:cumulants}
	Let $X\sim\mathcal{B}^{(s)}(n,p)$ and let $x_i=\sum_{l=1}^{s}l^ip^lq^{s-l}=E(X_1^i)$. Then, for every $r\geq1$,
	\begin{equation}
		\kappa_r(X)=n\sum_{k=1}^{r}(-1)^{k-1}(k-1)!\,B_{r,k}\big(x_1,x_2,\ldots,x_{r-k+1}\big), \quad \text{ for } p\neq q, \label{eq:cum}
	\end{equation}
	and
	\begin{equation}
		\kappa_r(X)=n\sum_{k=1}^{r}(-1)^{k-1}(k-1)!\,p^{ks}\,B_{r,k}\Big(\sum_{l=1}^{s}l,\sum_{l=1}^{s}l^2,\ldots,\sum_{l=1}^{s}l^{r-k+1}\Big), \quad \text{ for } p=q.
	\end{equation}
\end{theorem}
\begin{proof}
	Since $K_X=n\ln M_{X_1}$ with $M_{X_1}(t)=1+\sum_{m\geq1}x_mt^m/m!$, the claim follows from the logarithmic polynomial identity (see \cite[Chap.~3]{com1})
	\begin{equation*}
		\ln\Big(1+\sum_{m\geq1}x_m\frac{t^m}{m!}\Big)=\sum_{r\geq1}\Big(\sum_{k=1}^{r}(-1)^{k-1}(k-1)!\,B_{r,k}(x_1,\ldots,x_{r-k+1})\Big)\frac{t^r}{r!},
	\end{equation*}
	obtained by composing the series $\ln(1+u)=\sum_{k\geq1}(-1)^{k-1}u^k/k$ with $u=\sum_m x_mt^m/m!$ through \eqref{eq:bellgf}. The case $p=q$ follows again from the homogeneity of degree $k$ of $B_{r,k}$.
\end{proof}

Since the one-die moments $x_i$ admit the closed form \eqref{eq:mr} of Theorem~\ref{thm:moments-sB} in terms of the extended Eulerian numbers $A(r,k,s)$, the substitution of \eqref{eq:mr} into \eqref{eq:cum} provides a fully explicit expression of all the cumulants of the bi$^{s}$nomial distribution in terms of $(s,p,q,n)$.

The first cumulants are $\kappa_1(X)=E(X)$ and $\kappa_2(X)=V(X)$, recovering the corollary above; moreover $\kappa_3(X)=n\big(x_3-3x_1x_2+2x_1^3\big)$ and $\kappa_4(X)=n\big(x_4-4x_1x_3-3x_2^2+12x_1^2x_2-6x_1^4\big)$. In the uniform case $p=q$, one gets
\begin{equation}
	\kappa_3(X)=0,\qquad \kappa_4(X)=-\,\frac{n\,s(s+2)(s^2+2s+2)}{120},
\end{equation}
so that the excess kurtosis is
\begin{equation}
	\gamma_2-3=\frac{\kappa_4(X)}{\kappa_2(X)^2}=-\,\frac{6(s^2+2s+2)}{5sn(s+2)},
\end{equation}
in accordance with \eqref{eq:kurt-bis}. In general, $\gamma_1=\kappa_3(X)/\kappa_2(X)^{3/2}=O(n^{-1/2})$ and $\gamma_2-3=\kappa_4(X)/\kappa_2(X)^{2}=O(n^{-1})$: the linearity of the cumulants in $n$ thus gives a direct quantitative explanation of the asymptotic normality established in Section~6.

\section{Going to the Poisson Distribution}

Classically, the Poisson distribution is a limiting case of the binomial distribution which arises when the number of trials $ n $ increases indefinitely while the parameter $\lambda$, which is the expected value of the number of successes over the trials, remains constant. Using the same approach, we do the same for the bi$^{s}$nomial distribution.

\begin{lemma}\label{lem:1}
	Let $k$ be a fixed nonnegative integer. Then
	\begin{equation}
		\lim_{n\rightarrow\infty}\dfrac{{n\choose k}}{n^k}=\dfrac{1}{k!}.
	\end{equation}
\end{lemma}

\begin{proof}
	We have	$ {n\choose k}=\frac{n(n-1)\cdots(n-k+1)}{k!} $; hence
	$ \lim\limits_{n\rightarrow\infty}\dfrac{{n\choose k}}{n^k}=\lim\limits_{n\rightarrow\infty}\dfrac{1}{k!}\prod_{j=0}^{k-1}\Big(1-\frac jn\Big)=\dfrac{1}{k!} .$
\end{proof}
\begin{lemma}\label{2} Let $ s\geq1 $ and $ k\geq0 $ be fixed integers. We have
	\begin{equation}
		\lim\limits_{n\rightarrow\infty} \dfrac{{n \choose k}_s}{n^k}=\dfrac{1}{k!}.
	\end{equation}
\end{lemma}
\begin{proof}
	We use the formula of Belbachir and Benmezai \cite{BBBen14}:
	\begin{equation}
		{n \choose k }_s =(-1)^k \sum\limits_{j_1+j_2+\cdots+j_s=k} {n \choose j_1} {n \choose j_2}\cdots {n \choose j_s} a^{j_1+2j_2+\cdots+sj_s},
	\end{equation}
	with $ a=\exp\left( -2\pi i/(s+1)\right)  $, where $ i^2=-1. $ 
	Since the sum is finite, we may interchange limit and summation:
	\begin{equation*}
		\begin{split}
			\lim\limits_{n\rightarrow +\infty} \frac{{ n \choose k}_s}{n^k}&= (-1)^k \sum\limits_{j_1+\cdots+j_s=k} \lim\limits_{n\rightarrow +\infty}\frac{{n \choose j_1}}{n^{j_1}} \frac{{n \choose j_2}}{n^{j_2}}\cdots \frac{{n \choose j_s}}{n^{j_s}} a^{j_1+2j_2+\cdots+sj_s}
			\\&=(-1)^k \sum\limits_{j_1+\cdots+j_s=k} \frac{1}{j_1!} \frac{1}{j_2!}\cdots \frac{1}{j_s!} a^{j_1+2j_2+\cdots+sj_s}
			\\&= \frac{(-1)^k}{k!} \sum\limits_{j_1+\cdots+j_s=k} \binom{k}{j_1,\ldots,j_s}a^{j_1+2j_2+\cdots+sj_s}\\&= \frac{(-1)^k}{k!}(a+a^2+\cdots+a^s)^k\\&= \frac{(-1)^k}{k!} \left( \frac{a-a^{s+1}}{1-a}\right) ^k.
		\end{split}
	\end{equation*}
	Since $ a^{s+1}=\exp\left( -2\pi i\right) =1$, we get $a+a^2+\cdots+a^s=\frac{a-a^{s+1}}{1-a}=\frac{a-1}{1-a}=-1$, and then
	\begin{equation*}
		\lim\limits_{n\rightarrow +\infty} \frac{{ n \choose k}_s}{n^k}=\frac{(-1)^k}{k!} (-1)^k
		=\frac{1}{k!}.
	\end{equation*}
\end{proof}

We shall use the following classical result, whose proof can be found in \cite{car}.
\begin{theorem}[Implicit Function Theorem]\label{4}
	Assume that $ \Omega $ is an open subset of $ \mathbb{R}^2 $ and that $ f : \Omega\rightarrow \mathbb{R} $ is a function of class $ C^k $. Assume also that $ (a,b) $ is a point in $ \Omega $ such that 
	$$ f(a,b)=0 \qquad\qquad \text{ and }\qquad\qquad \frac{\partial f}{\partial y}(a,b)\neq 0. $$
	Then there exist an open neighborhood $ U $ of $ (a,b) $ in $ \mathbb{R}^2 $, an open interval $ I $ of $ \mathbb{R} $ which contains $ a $, and a function $ g:I\rightarrow\mathbb{R} $ of class $ C^k $ such that, for any $ (x,y)\in U $, we have 
	$$ f(x,y)=0\Leftrightarrow y=g(x). $$
	Moreover, for all $ x \in I$, we have
	$$ g'(x)=-\dfrac{\dfrac{\partial f}{\partial x}(x,g(x))}{\dfrac{\partial f}{\partial y}(x,g(x))}. $$
\end{theorem}

\begin{theorem}\label{thm:poisson}
	Let $s\geq2$ be a fixed integer. Let $(p_n)_n$ and $(q_n)_n$ be two sequences of real numbers in the interval $]0,1[$ satisfying, for each $n$, the condition
	$$ p_n^s + p_n^{s-1}q_n + \cdots + q_n^s = 1. $$
	Let $(X_n)_n$ be a sequence of random variables, where each $X_n$ takes its values in the set $\{0, 1, \dots,sn\}$ and follows a bi$^s$nomial distribution $\mathcal{B}^{(s)}(n, p_n)$. Assume the sequence of expectations converges to a finite, positive limit $\lambda_s$:
	\begin{equation}
		\lim_{n \to \infty} E(X_n) = \lim_{n \to \infty} np_n \frac{(s+1)p_n^s - 1}{p_n - q_n} = \lambda_s > 0.
	\end{equation}
	Then the sequence of random variables $(X_n)_n$ converges in law to a Poisson distribution with parameter $\lambda_s$; in other words, for any integer $k \ge 0$,
	\begin{equation}
		\lim_{n \to \infty} \mathbb{P}(X_n=k) = \lim_{n \to \infty} {n \choose k}_s p_n^k q_n^{sn-k} = \frac{\lambda_s^k}{k!}e^{-\lambda_s}.
	\end{equation}
\end{theorem}
\begin{proof}
	The main goal is to evaluate the limit 
	\begin{equation}
		\lim\limits_{n\rightarrow\infty} {n \choose k}_s p_n^k q_n^{sn-k}, \text{ with } p_n^s+p_n^{s-1}q_n+\cdots+q_n^s=1. \label{eq:2}
	\end{equation} 
	
	\emph{Step 0: $p_n\rightarrow0$ and $q_n\rightarrow1$.} The admissible pairs $(p,q)$ form a curve on which the one-die expectation $e(p):=p\frac{(s+1)p^s-1}{p-q}$ is continuous and positive, and $E(X_n)=n\,e(p_n)$. If $p_n$ did not tend to $0$, then along a subsequence $p_n\geq\varepsilon$ for some $\varepsilon>0$, hence $e(p_n)\geq c(\varepsilon)>0$ by continuity and compactness, and $E(X_n)\geq n\,c(\varepsilon)\rightarrow\infty$, contradicting $E(X_n)\rightarrow\lambda_s<\infty$. Therefore $p_n\rightarrow0$ and, by the constraint, $q_n^s+q_n^{s-1}p_n+\cdots\rightarrow1$ forces $q_n\rightarrow1$.
	
	\emph{Step 1: expansion of $q_n$ in terms of $p_n$.} Consider the function of two real variables $f(x,y)=x^{s+1}-y^{s+1}-x+y$, so that the constraint reads $f(p_n,q_n)=0$. We have
	\begin{equation*}
		\dfrac{\partial f}{\partial y}(x,y)=-(s+1)y^{s}+1,\qquad\text{hence}\qquad 	\dfrac{\partial f}{\partial y}(0,1)=-s\neq0.
	\end{equation*} 
	Since in addition $ f(0,1)=0 $, by applying Theorem \ref{4}, there are two open intervals $ I $ and $ J $, with $ 0\in I $ and $ 1\in J $, and a function $ g:I\rightarrow J $ of class $ C^1 $ such that
	\begin{equation*}
		\text{for all }(x,y)\in I \times J, \quad f(x,y)=0\Leftrightarrow y=g(x).
	\end{equation*}
	Moreover, since $ f $ is of class $ C^\infty $, $ g $ is also of class $ C^\infty $ on $ I $. In particular, $g$ admits a Taylor expansion of any order at $0$ (Taylor--Young); at order two,
	\begin{equation}
		g(x)= g(0)+\dfrac{g'(0)}{1!}x+\dfrac{g''(0)}{2!}x^2+o(x^2). \label{eq14}
	\end{equation}
	We already know that $ g(0)=1 $. To calculate $ g'(0) $ and $ g''(0) $, we differentiate twice the identity $f(x,g(x))=0$, that is,
	\begin{equation*}
		x^{s+1} -g(x)^{s+1}-x+g(x)=0.
	\end{equation*}
	We differentiate once and evaluate at $ x=0$ to find
	\begin{equation*}
		(s+1)x^s-(s+1)g'(x)g(x)^{s}-1+g'(x)=0\Rightarrow g'(0)=\dfrac{-1}{s}.
	\end{equation*}
	We differentiate a second time and evaluate at $ x=0 $ to find
	\begin{equation*}
		\begin{split}
			&	(s+1)sx^{s-1}-(s+1) g''(x)g(x)^{s}-(s+1)s g'(x)^2g(x)^{s-1}+g''(x)=0\\&\Rightarrow g''(0)=-\dfrac{s+1}{s^2},
		\end{split}
	\end{equation*}
	where the first term vanishes at $x=0$ precisely because $s\geq2$. Thus, since $p_n\rightarrow0$ (Step~0) allows us to evaluate \eqref{eq14} along the sequence, we get
	\begin{equation}
		q_n=1-\frac{p_n}{s}-\frac{(s+1)p_n^2}{2s^2}+o(p_n^2). \label{eq:4}
	\end{equation}
	
	\emph{Step 2: $np_n\rightarrow\lambda_s$.} As $p_n\rightarrow0$ and $q_n\rightarrow1$, we have
	\begin{equation*}
		\frac{(s+1)p_n^s-1}{p_n-q_n}\longrightarrow\frac{-1}{0-1}=1,
	\end{equation*}
	so that $E(X_n)=np_n\frac{(s+1)p_n^s-1}{p_n-q_n}\rightarrow\lambda_s$ yields $np_n\rightarrow\lambda_s$, i.e.\ $p_n=\frac{\lambda_s}{n}\big(1+o(1)\big)$.
	
	\emph{Step 3: $q_n^{sn-k}\rightarrow e^{-\lambda_s}$.} By \eqref{eq:4}, $\ln q_n=-\frac{p_n}{s}+O(p_n^2)$, hence
	\begin{equation*}
		(sn-k)\ln q_n=-np_n+O\big(np_n^2\big)+O(p_n)\longrightarrow-\lambda_s,
	\end{equation*}
	since $np_n^2=(np_n)\,p_n\rightarrow\lambda_s\cdot0=0$; therefore $q_n^{sn-k}\rightarrow e^{-\lambda_s}$.
	
	\emph{Step 4: conclusion.} Writing
	\begin{equation*}
		{n \choose k}_s p_n^k q_n^{sn-k}=\frac{{n \choose k}_s}{n^k}\,(np_n)^k\, q_n^{sn-k},
	\end{equation*}
	Lemma \ref{2} together with Steps 2 and 3 gives the convergence to the Poisson distribution:
	\begin{equation*}
		\lim\limits_{n\rightarrow +\infty} { n \choose k}_s p_n^k q_n^{sn-k}  =\frac{1}{k!}\,\lambda_s^k\, e^{-\lambda_s} .
	\end{equation*}  
\end{proof}

\begin{remark}
	For $s=1$, the constraint gives exactly $q_n=1-p_n$ (and $g''(0)=0$, so the expansion \eqref{eq:4} does not hold as stated); the statement then reduces to the classical convergence of the binomial distribution to the Poisson distribution, and Steps~2--4 apply verbatim.
\end{remark}

The Poisson distribution $  \mathcal{P}(\lambda_s) $ is the limiting distribution of the bi$^{s}$nomial distribution as $ n $ tends to infinity. We note that the expectation and the variance of the bi$^{s}$nomial distribution converge towards those of the Poisson distribution as $n$ tends to infinity: indeed, $V(X_n)=np_nq_n\frac{1-(s+1)^2(p_nq_n)^s}{(p_n-q_n)^2}\rightarrow\lambda_s$, in accordance with $E=V=\lambda$ for the Poisson distribution.

\section{Going to a Normal Distribution}

In the limiting case where the number of observations $n$ becomes infinite while the success probability $p$ is fixed (so that $np\gg1$), the binomial distribution is approximated by a normal distribution having the same expectation $ \mu $ and the same variance $ \sigma^2 $. We apply this approach to the bi$^{s}$nomial distribution and show that it converges to the normal distribution.

\begin{theorem}\label{thm:clt}
	Let $ X\sim \mathcal{B}^{(s)}(n,p) $ be a bi$^{s}$nomial random variable with parameters $ n $ and $ p $, with $p\neq q$, and set
	$$ \mu_s=E(X)=np\frac{(s+1)p^{s}-1}{(p-q)},\qquad \sigma^2_s=V(X)=npq\frac{ 1-(s+1)^2(pq)^s}{(p-q)^2}. $$
	Then the standardized variable $ Z=\frac{X-\mu_s}{\sigma_s} $ converges in distribution, as $n\rightarrow\infty$, to the standard normal distribution $N(0,1)$; in other words, $X$ has, for large $n$, an approximate normal distribution with mean $\mu_s$ and variance $\sigma_s^2$.
\end{theorem}

\begin{proof}
	\underline{\textbf{The MGF method}}\\
	
	The moment generating function of the random variable $ X $ evaluates to 
	
	$ M_X(t)=\sum\limits_{k=0}^{s n}e^{t k}{n\choose k}_s p^k q^{sn-k}=\left( \sum\limits_{k=0}^{s} \left( e^t p\right) ^k q^{s-k}\right)^n $. 
	
	Let $ Z=\frac{1}{\sigma_s}\left( X-E(X) \right)  $. Below we derive the mgf of $ Z $, which is given by
	\begin{equation*}
		M_Z(t)=E(e^{tZ})=\left( \sum\limits_{k=0}^{s} \exp\left( \left( k-\frac{E(X)}{n}\right) \frac{t}{\sigma_s}\right)  p^k q^{s-k}\right)^n.
	\end{equation*}
	
	The Taylor expansion with Lagrange remainder of $ \exp\left( \left( k-\frac{E(X)}{n}\right) \frac{t}{\sigma_s}\right) $ gives
	\begin{equation*}
		\begin{split}
			\exp\left( \left( k-\frac{E(X)}{n}\right) \frac{t}{\sigma_s}\right)=&1+\left( k-\frac{E(X)}{n}\right) \frac{t}{\sigma_s}+\frac{1}{2!}\left( k-\frac{E(X)}{n}\right)^2\left(\frac{t}{\sigma_s}\right)^2+\\&\frac{1}{3!}\left( k-\frac{E(X)}{n}\right)^3 \left(\frac{t}{\sigma_s}\right)^3+\frac{1}{4!}\left( k-\frac{E(X)}{n }\right)^4\left( \frac{t}{\sigma_s}\right) ^4e^{\psi_k(n)},
		\end{split}
	\end{equation*}
	where $ \psi_k(n) $ is a number between 0 and $ \left( k-\frac{E(X)}{n }\right) \frac{t}{\sigma_s} $. Since $k$ ranges over the finite set $\{0,\dots,s\}$, $\frac{E(X)}{n}=E(X_1)$ is bounded and $\sigma_s\rightarrow\infty$, we have $ \max_{0\leq k\leq s}|\psi_k(n)|\rightarrow0 $ as $ n\rightarrow\infty $, uniformly for $t$ in compact sets.
	Now substituting this equation in the last expression for $ M_Z(t) $, we have 
	\begin{equation}
		\begin{split}
			M_Z(t)&=\Biggl(\sum\limits_{k=0}^{s} p^k q^{s-k}+\sum\limits_{k=0}^{s}\left( k-\frac{E(X)}{n }\right)p^k q^{s-k} \left( \frac{t}{\sigma_s}\right) +\frac{1}{2!}\sum\limits_{k=0}^{s}\left( k-\frac{E(X)}{n }\right)^2 p^k q^{s-k} \left( \frac{t}{\sigma_s}\right) ^2\\&+\frac{1}{3!}\sum\limits_{k=0}^{s}\left( k-\frac{E(X)}{n }\right)^3 p^k q^{s-k} \left( \frac{t}{\sigma_s}\right) ^3+\frac{1}{4!}\sum\limits_{k=0}^{s}\left( k-\frac{E(X)}{n }\right)^4p^kq^{s-k}\left( \frac{t}{\sigma_s}\right) ^4e^{\psi_k(n)}\Biggr)^n,
		\end{split}
		\label{eq24}
	\end{equation}
	where $ \sum\limits_{k=0}^{s} p^k q^{s-k}=1 $ and $\sum\limits_{k=0}^{s}\left( k-\frac{E(X)}{n }\right)p^k q^{s-k} =0$, since 
	$$ \sum\limits_{k=0}^{s} kp^k q^{s-k}-\left( \sum\limits_{k=0}^{s}p^k q^{s-k}\right) \frac{E(X)}{n}= E(X_1)- \frac{E(X)}{n}=E(X_1)-E(X_1)=0.$$
	Moreover, $\sum\limits_{k=0}^{s}\left( k-\frac{E(X)}{n }\right)^2\left( \frac{1}{\sigma_s}\right) ^2 p^k q^{s-k} = \tfrac{1}{n}$, since, using $\frac{E(X)}{n}=E(X_1)$,
	\begin{equation*}
		\begin{split}
			\sum\limits_{k=0}^{s}\left( k-\frac{E(X)}{n}\right)^2 p^k q^{s-k}&=\sum\limits_{k=0}^{s}\left( k^2-2kE(X_1)+E(X_1)^2\right) p^k q^{s-k}\\&=\sum\limits_{k=0}^{s} k^2p^k q^{s-k}-2E(X_1)\sum\limits_{k=0}^{s}kp^k q^{s-k}+E(X_1)^2\sum\limits_{k=0}^{s}p^k q^{s-k}\\&=E(X_1^2)-2E(X_1)^2+E(X_1)^2\\&=E(X_1^2)-E(X_1)^2\\&= V(X_1)=\frac{V(X)}{n}=\frac{\sigma_s^2}{n}.		
		\end{split}
	\end{equation*}
	Writing $ \sigma_s^2=n\Big( \sum\limits_{k=0}^{s}k^2p^kq^{s-k} -E(X_1)^2 \Big)=nV(X_1)$ and substituting these identities in the last expression for $ M_Z(t) $ in \eqref{eq24}, we have
	\begin{equation*}
		\begin{split}
			M_Z(t)&= \Biggl(1 +\frac{t^2}{2n}+\frac{1}{3!}  \left( \frac{\sum\limits_{k=0}^{s}\left( k-E(X_1)\right)^3p^k q^{s-k}}{n^\frac{3}{2}\,V(X_1)^{\frac{3}{2}}}\right)t^3+\frac{1}{4!}\left( \frac{\sum\limits_{k=0}^{s}\left( k-E(X_1)\right)^4p^k q^{s-k}\,e^{\psi_k(n)}}{n^2\,V(X_1)^2}\right)t^4\Biggr)^n.
		\end{split}
	\end{equation*}
	
	The above equation may be written as 
	\begin{equation*}
		M_Z(t)= \Biggl(1 +\frac{t^2}{2n}+\frac{\Psi(n)}{n}\Biggr)^n,
	\end{equation*}
	where
	\begin{equation*}
		\begin{split}
			\Psi(n)&= \frac{1}{3!}  \left( \frac{\sum\limits_{k=0}^{s}\left( k-E(X_1)\right)^3p^k q^{s-k}}{\sqrt{n}\,V(X_1)^{\frac{3}{2}}}\right)t^3+\frac{1}{4!}\left( \frac{\sum\limits_{k=0}^{s}\left( k-E(X_1)\right)^4p^k q^{s-k}\,e^{\psi_k(n)}}{n\,V(X_1)^2}\right)t^4.
		\end{split}
	\end{equation*}
	
	The numerators above are central moments of a distribution with finite support, hence bounded; since moreover $ \max_k|\psi_k(n)|\rightarrow0 $ as $ n\rightarrow\infty $, it follows that $ \lim\limits_{n\rightarrow\infty}\Psi(n)=0 $ for every fixed value of $ t $. Using the standard fact that $\big(1+\frac{a_n}{n}\big)^n\rightarrow e^{a}$ whenever $a_n\rightarrow a$, we obtain
	\begin{equation*}
		\lim\limits_{n\rightarrow \infty} M_Z(t)=e^{t^2/2}
	\end{equation*}
	for all real values of $ t $, which is the moment generating function of the standard normal distribution $N(0,1)$. By the continuity theorem for moment generating functions (Curtiss \cite{Cur}), $Z$ converges in distribution to $N(0,1)$.
\end{proof}

\begin{remark}
	For $p=q$, the same conclusion holds with $\mu_s=sn/2$ and $\sigma_s^2=sn(s+2)/12$; indeed, $X$ being a sum of $n$ i.i.d.\ bounded random variables, the convergence also follows from the classical L\'evy--Lindeberg central limit theorem. The value added by Theorem~\ref{thm:clt} lies in the explicit closed forms of $\mu_s$ and $\sigma_s^2$ in the parametrization $(p,q)$, and in the self-contained proof by generating functions; the rate $\gamma_2-3=O(n^{-1})$ obtained from the cumulants in Section~4.3 quantifies this convergence at the level of the fourth moment.
\end{remark}

We can conclude that, as $ n\rightarrow\infty $, the random variable $ Z=\frac{X-E(X)}{\sigma_s} $ has the standard normal as its limiting distribution; or equivalently, that the bi$^{s}$nomial random variable $ X $ has, for large $ n $, an approximate normal distribution with mean $ \mu_s=np\frac{(s+1)p^{s}-1}{(p-q)} $ and variance $ \sigma_s^2=npq\frac{ 1-(s+1)^2(pq)^s}{(p-q)^2} . $

\AuthorAddresses
\end{document}